\documentclass{amsart}
\usepackage{amssymb}
\usepackage{amsthm}
\usepackage{mathtools}
\usepackage{enumerate}
\usepackage{microtype}
\usepackage[utf8]{inputenc}

\theoremstyle{plain}
\newtheorem{theorem}{Theorem}[section]
\newtheorem{lemma}[theorem]{Lemma}
\newtheorem{proposition}[theorem]{Proposition}
\newtheorem{corollary}[theorem]{Corollary}

\theoremstyle{definition}

\newtheorem{example}[theorem]{Example}

\theoremstyle{remark}
\newtheorem*{remark}{Remark}

\newcommand{\bd}{\partial}
\newcommand{\eps}{\varepsilon}
\newcommand{\psh}{\mathcal{PSH}}
\newcommand{\suchthat}{\mathrel{;}}
\DeclarePairedDelimiter\abs{\lvert}{\rvert}
\DeclarePairedDelimiter\set{\{}{\}}

\title{Quasibounded plurisubharmonic functions}
\author{Mårten Nilsson}
\address{Center for Mathematical Sciences\\
  Lund University\\
  Box 118, SE-221 00 Lund, Sweden}
\email{marten.nilsson@math.lth.se}

\author{Frank Wikström}
\email{frank.wikstrom@math.lth.se}


\begin{document}

\maketitle
\begin{abstract}
   We extend the notion of quasibounded harmonic functions to the plurisubharmonic setting. As an application, using the theory of Jensen measures, we show that certain generalized Dirichlet problems with unbounded boundary data admit unique solutions, and that these solutions are continuous outside a pluripolar set.
\end{abstract}
\section{Introduction}

Quasibounded harmonic functions, i.e. harmonic functions that can be
approximated monotonely from below by \emph{bounded} harmonic functions, were
first introduced by Parreau~\cite{parreau}. In this work, we will develop an
analogous theory of quasibounded plurisubharmonic functions, and as an
application solve certain Dirichlet problems for unbounded maximal
plurisubharmonic functions.

As motivation, let us begin with a few simple examples.

\begin{example}
    Let $\Omega = \mathbb{D} \setminus \{ 0 \}$ be the punctured unit
    disc, and let $u(z) = -\log \abs{z}$.

    Then $u$ is harmonic (thus subharmonic) on $\Omega$ and if $v \le u$ is
    bounded and subharmonic on $\Omega$, by Riemann's removable singularities
    theorem, $v$ extends to a subharmonic function on $\mathbb{D}$.
    Consequently, by the maximum principle, $v \le 0$. In particular, $u$
    cannot be written as an increasing limit of bounded subharmonic functions.
\end{example}

A slightly more interesting example, where the domain is regular, is the following:

\begin{example}\label{ex:poisson-kernel}
    Let $\Omega = \mathbb{D}$ and let $u(z) = \frac{1-\abs{z}^2}{\abs{1-z}^2}$,
    i.e. the Poisson kernel for $\mathbb{D}$ with pole at $z=1$.

    If $v \le u$ is a bounded subharmonic function, then
 $\limsup_{z \to \zeta}
    v(z) \le 0$ for every $\zeta \in \partial\mathbb{D} \setminus \{ 1 \}$,
    so the extended maximum principle (see e.g.~\cite[Thm 3.6.9]{ransford})
    implies that $v \le 0$. Again, we see that $u$
    cannot be written as an increasing limit of bounded subharmonic functions.

This second example can be easily adapted to several variables. Let
\[
    u(z, w) = \frac{1-\abs{z}^2}{\abs{1-z}^2}
\]
and view $u$ as a function on the unit ball $\mathbb{B}$. Then $u$ is maximal
plurisubharmonic (even pluriharmonic) and every bounded plurisubharmonic function $v \le u$ must be
negative on $\mathbb{D} \times \{ 0 \}$.
\end{example}

Any bounded plurisubharmonic function is of course quasibounded, but in general, it is currently not
an easy task to determine whether a particular unbounded plurisubharmonic function is quasibounded or not.
owever, in 1974, Arsove and Leutwiler~\cite{arsove} gave necessary and sufficient conditions for quasiboundedness
in a generalized sense for a large set of functions, including non-negative subharmonic, harmonic and
superharmonic functions, extending a previous characterization for the harmonic case by Parreau~\cite{parreau}
and Yamashita~\cite{yamashita}. The central idea was to introduce a certain operator $S$ on the set of non-negative
functions on $\Omega$ admitting a superharmonic majorant, providing a reformulation of quasiboundedness since
for a harmonic function $h$,
\[
    h \text{ is quasibounded} \iff S(h)=0.
\]
In Section~\ref{sec:tame} and~\ref{sec:quasibounded}, we generalize their techniques to the plurisubharmonic setting.

Section~\ref{sec:dirichlet-problem} concerns the Dirichlet problem for the complex Monge--Ampère equation.
This problem is very well-studied, at least when the boundary data is a continuous function. The
problem amounts to solving the partial differential equation for $u \in \psh(\Omega)$,
\[
    \begin{cases}
        (dd^c u)^n = \mu, & \text{on $\Omega$} \\
        u = \phi, & \text{on $\partial\Omega$},
    \end{cases}
\]
where $\mu$ is a positive measure, and both the complex Monge--Ampère operator
$(dd^c u)^n$ and the boundary values of $u$ must be suitably interpreted.
Already in~1959, Bremermann~\cite{bremermann} showed that the problem has a
solution when $f \in C(\partial\Omega)$, $\mu = 0$ and $\Omega$ is a strictly
pseudoconvex domain with $C^2$ boundary. A few years later, in 1968,
Walsh~\cite{walsh} showed that Bremermann's solution is continuous on
$\bar\Omega$. Bedford and Taylor~\cite{bedford-taylor} generalized these results
to $\mu = f\,dV$ where $0 \le f \in C(\bar\Omega)$, again with a solution in
$C(\bar\Omega)$. B\l{}ocki~\cite{blocki} showed that $\Omega$ can be taken as
a $B$-regular domain or even as a hyperconvex domain under the additional assupmtion
that $\phi$ is the boundary value of \emph{some} plurisubharmonic function.

Many authors have generalized these results to measures $\mu$ that do not
necessarily have continuous density with respect to Lebesgue measure, but much less
has been done with weaker assumptions on the boundary data. In particular, we
are interested in solving the Dirichlet problem for maximal plurisubharmonic
functions (i.e.~$\mu = 0$) when $\phi$ is allowed to be a (continuous)
unbounded function. This problem comes up naturally when studying pluricomplex
Green functions with large singularity set, see for example
Lárusson--Sigur\dh{}sson~\cite{larusson}, Nguyen~\cite{dieu} and
Rashkovskii--Sigur\dh{}sson~\cite{rashkovski}. The main theorem of our paper, Theorem \ref{thm:dirichlet-problem}, gives sufficient conditions for the existence and uniqueness of a solution in this situation.

\section{Tame functions}\label{sec:tame}

Most of the results in Section~\ref{sec:tame} and~\ref{sec:quasibounded}
generalize results in~\cite{arsove} to the plurisubharmonic
setting, and the arguments needed are very similar. Notice that in general, our results are
weaker, 
since we cannot use linearity
(in particular, minimal plurisuperharmonic functions are not necessarily
plurisubharmonic) or Harnack's theorem (the limit of an increasing sequence of maximal plurisubharmonic functions is not necessarily a maximal plurisubharmonic function, since it may fail to be upper semicontinuous).

Let $\Omega \subset \mathbb{C}^n$ be a bounded domain and let
\[
    \mathcal{M} = \set{ f\colon \Omega \to [0,+\infty] \suchthat \exists u \in -\psh(\Omega), f \le u },
\]
i.e.\ the set of non-negative functions on $\Omega$ admitting a
\emph{plurisuperharmonic} majorant. If $\lambda \ge 0$, we let $(f-\lambda)^+ =
\max \{ f-\lambda, 0 \}$. Note that $f \in \mathcal{M}$ implies that
$(f-\lambda)^+ \in \mathcal{M}$ for every $\lambda \ge 0$.

Given $f \in \mathcal{M}$, we define the \emph{reduced} function $R_\lambda f$ by
\[
    (R_\lambda f)(z) = \inf \set{ v(z) \suchthat v \in -\psh(\Omega), v \ge (f-\lambda)^+},
\]
i.e. the infimum of all plurisuperharmonic majorants of $(f-\lambda)^+$. In
general, $R_\lambda$ is not plurisuperharmonic, since it may fail to be lower
semicontinuous, but if we define
\[
    S_\lambda f = (R_\lambda f)_*
\]
as the lower semicontinuous regularization of $R_\lambda$, then $S_\lambda \in
-\psh(\Omega)$ and $S_\lambda f = R_\lambda f$ outside a pluripolar set.

Next, note that $S_\lambda f$ is decreasing in $\lambda$ (for $\lambda \ge 0$).
Let
\[
    S_\infty f = \lim_{\lambda \to \infty} S_\lambda f.
\]
Again, $S_\infty$ might not be plurisuperharmonic, but
\[
    Sf = (S_\infty f)_*
\]
certainly is. This construction gives us an operator $S$ from $\mathcal{M}$
to the set of non-negative plurisuperharmonic functions.

Finally, we say that $f \in \mathcal{M}$ is \emph{tame} if $Sf \equiv 0$, and
\emph{singular} if $Sf \equiv f$, respectively.

\subsection{Basic properties of the operator $S$}

Let us collect some straight-forward but useful properties of the operator $S$
in the following proposition.

\begin{proposition}\label{prop:basic}
    Let $f, g \in \mathcal{M}$ and $\alpha \ge 0$. Then:
    \begin{enumerate}[(1)]
        \item $S(\alpha f) = \alpha S(f)$,
        \item $S(f + g) \le Sf + Sg$,
        \item if $f \le g$, then $Sf \le Sg$,
        \item $S( \min(f,g) ) \le \min( Sf, Sg )$.
    \end{enumerate}
\end{proposition}

\begin{proof}
    Statement~(4) follows from~(3) which in turn follows directly from the
    definition of $S$.

    To prove~(1), assume that $u \in -\psh(\Omega)$ with $f \le u + \lambda$ for
    some $\lambda \ge 0$. Then $\alpha f \le \alpha u + \alpha\lambda$, and
    thus $S_{\alpha\lambda}(\alpha f) \le \alpha u$. Taking the infimum over all
    such $u$, we see that
    \[
        S_{\alpha\lambda}(\alpha f) \le \alpha S_\lambda f
    \]
    and letting $\lambda\to \infty$
    \[
        S(\alpha f) \le \alpha Sf.
    \]
    If $\alpha > 0$, $Sf = S(\alpha f/\alpha) \le S(\alpha f)/\alpha$, and
    combining the two inequalities gives $S(\alpha f) = \alpha S(f)$. If $\alpha =
    0$, (1) holds trivially.

    Finally, we prove~(2) using a similar argument: if $u$ and $v$ are non-negative
    plurisuperharmonic functions with $f \le u + \lambda$ and $g \le v + \lambda$
    for some $\lambda \ge 0$, then
    \[
        f + g \le u + v + 2\lambda
    \]
    so, taking the infimum over all such $u$ and $v$
    \[
        S_{2\lambda}(f+g) \le S_\lambda f + S_\lambda g
    \]
    and (2) follows by letting $\lambda \to \infty$.
\end{proof}

We can also establish the following uniqueness result.

\begin{lemma}\label{lemma:uniqueness}
    If $f, g \in \mathcal{M}$ such that $f \le g$ outside a
    pluripolar subset of $\Omega$, then $Sf \le Sg$ everywhere on
    $\Omega$.

    In particular, if $f = g$ quasi-everywhere, then $Sf \equiv Sg$.
\end{lemma}

\begin{proof}
    Assume that $f \le g$ outside a pluripolar subset $E \subset \Omega$. Then
    there is a non-negative plurisuperharmonic function $\phi$ such that
    $\phi(z) = +\infty$ for every $z \in E$.

    For every $\eps > 0$, we have $f \le g + \eps \phi$ on $\Omega$, so
    \[
        Sf \le Sg + \eps S\phi.
    \]
    If we let $\eps \searrow 0$, it follows that $Sf \le Sg$ on $\Omega
    \setminus E$, but since $Sf$ and $Sg$ are plurisuperharmonic, the inequality
    holds everywhere.
\end{proof}

Let us also formulate the following observation that will be useful
in section~\ref{sec:quasibounded}.

\begin{lemma}\label{lemma:add_tame}
    If $f, q \in \mathcal{M}$ and $q$ is tame, then
    $S(f+q) = Sf$.
\end{lemma}

\begin{proof}
    Since $f \le f + q$, it follows from Proposition~\ref{prop:basic} that
    \[
        Sf \le S(f+q) \le Sf + Sq = Sf + 0. \qedhere
    \]
\end{proof}

\section{Quasibounded plurisubharmonic functions}\label{sec:quasibounded}

We now specialize the preliminary work on tame functions in
Section~\ref{sec:tame} to plurisubharmonic functions.

%
%


Assume that $u \in \mathcal{M}$ is plurisubharmonic. Then, for every $\lambda
\ge 0$, $S_\lambda u$ is plurisuperharmonic (and so is $Su$).
In particular, for each $n = 1, 2, \ldots$, $u_n = u - S_n u$ is an upper bounded
plurisubharmonic function, since by definition
\[
    u \le S_n u + n.  
\]

\begin{theorem}\label{thm:qb-psh}
    Let $0 \le u \in \mathcal{M} \cap \psh(\Omega)$. If $u$ is tame,
    then $u$ is the limit of an increasing
    sequence of upper bounded plurisubharmonic functions on $\Omega$ outside a pluripolar set.
\end{theorem}

\begin{proof}
    If $u$ is tame, then $Su = 0$, so
    \[
        u_n = u - S_n u
    \]
    is a sequence of bounded plurisubharmonic functions on $\Omega$ with $u_n
    \nearrow u$ quasi-everywhere.
    %
\end{proof}

Note that in the linear setting, the analogue of Theorem~\ref{thm:qb-psh}
is an equivalence that holds at all points of $\Omega$, cf.~\cite[Theorem 3.1]{arsove}. At this point, we do not know whether a tame plurisubharmonic function $u$ on $\Omega$ satisfies
\[
u(z) = \sup \{ v(z) : v\in \psh(\Omega), \sup v < +\infty \}
\]
\emph{everywhere} on $\Omega$ instead of just quasi-everywhere. We will return to this question in Section~\ref{sec:dirichlet-problem}.

Let us conclude this section with a necessary and sufficient condition for
tameness.

\begin{theorem}\label{thm:qb-psh2}
    Assume that $\psi\colon [-\infty, \infty] \to [0,+\infty]$ is any function
    such that $\psi(+\infty) = +\infty$ and
    \begin{equation}\label{eq:condition-phi}
        \lim_{t\to\infty} \frac{\psi(t)}{t} = +\infty.
    \end{equation}
    If $u$ is any function on $\Omega$ such that $\psi \circ u$ admits a
    plurisuperharmonic majorant, then $u^+$ is tame. In particular, if $u$
    is also plurisubharmonic, then $u^+$ is the limit of an increasing
    sequence of upper bounded plurisubharmonic functions on $\Omega$ outside a pluripolar set.

    Conversely, if $u^+$ is tame, there exists a nonnegative function $\phi$
    satisfying~\eqref{eq:condition-phi}, such that $\phi \circ u$ admits a
    plurisuperharmonic majorant. In fact, we may even take $\phi$ to be
    increasing and convex.
\end{theorem}

\begin{proof}
    From the assumption on the growth of $\psi$, it follows that for every $\eps
    > 0$, there is a positive integer $N$ such that $t \le \eps\psi(t)$  for all
    $t \ge N$. In particular $t \le \eps\psi(t) + N$ for all $t$. Hence
    \[
        u^+(z) \le \eps\psi(u(z)) + N.
    \]
    Since $\psi \circ u \in \mathcal{M}$, we get
    \[
        Su^+ \le \eps S(\psi \circ u) + S(N) = \eps S(\psi \circ u).
    \]
    Letting $\eps \to 0$ finishes the proof.

    For the converse, note that
    \[
        Su^+(z) = \lim_{n\to\infty} S_n u^+(z) = \lim_{n\to\infty} R_n u^+(z) = 0
    \]
    where each equality holds for all $z$ outside a pluripolar set.
    Fix a point $z_0 \in \Omega$ such that all three equalities above are satsified,
    and choose a sequence $n_k \nearrow \infty$ such that
    \[
        R_{n_k} u^+(z_0) < \frac{1}{2^k}.
    \]
    Hence, there is a sequence of nonnegative plurisuperharmonic functions $(v_k)$
    with
    \[
        (u - n_k)^+ \le v_k \quad\text{and}\quad v_k(z_0) < \frac{1}{2^k}.
    \]
    Define the function $\phi : [-\infty, +\infty] \to [0,+\infty]$ by
    \[
        \phi(t) = \sum_{k=1}^\infty (t - n_k)^+.
    \]
    Note that the sum is actually finite for $t < +\infty$ and that $\phi$ is
    a continuous piecewise linear with $\phi'(t) = k$ for $n_k < t < n_{k+1}$,
    so $\phi$ is an increasing convex function.

    If $n_m < t < +\infty$,
    \[
        \phi(t) > \sum_{k=1}^m (t-n_k) = mt - \sum_{k=1}^m n_k,
    \]
    so in particular
    \[
        \liminf_{t\to+\infty} \frac{\phi(t)}{t} = +\infty.
    \]
    Finally,
    \[
        \phi \circ u = \sum_{k=1}^\infty (u - n_k)^+  \le \sum_{k=1}^\infty v_k
    \]
    and since the right hand side converges at $z=z_0$, we see that $\sum v_k$
    is a plurisuperharmonic majorant of $\phi \circ u$.
\end{proof}

\begin{corollary}\label{cor:tame-sufficient}
    Let $u \in \mathcal{M}$ and $0 < \alpha < 1$. Then $u^\alpha$, $\log^+ u$
    and $u^\alpha \log^+ u$ are all tame.
\end{corollary}

\section{The Dirichlet problem for maximal unbounded plurisubharmonic functions}\label{sec:dirichlet-problem}

As an application of these concepts, we will show that
the Dirichlet problem for maximal plurisubharmonic functions with certain
unbounded boundary data admit solutions that are continuous outside a pluripolar set. We are not able to tackle the unbounded case in full generality, but if the
boundary function is assumed to be tame and continuous, we will show that the
Perron--Bremermann solution of the Dirichlet problem, as well as being continuous quasi-everywhere, is the unique maximal plurisubharmonic function with the prescribed boundary limits. As a tool,
we will use Jensen measures and a duality theorem by Edwards~\cite{edwards}. For
convenience of the reader, we include the necessary background here.

Let $X$ be a compact metric space, and let $\mathcal{F}$ be a convex cone of
upper semicontinuous and upper bounded real-valued functions on $X$
containing all the constants.

If $g\colon X \to \mathbb{R} \cup \set{ +\infty }$, we define
\[
    S^{\mathcal{F}} g(z) = \sup \set{ u(z) \suchthat u \in \mathcal{F}, u \le g }.
\]
If $z \in X$, we define a class of positive Radon measures
\[
    \mathcal{J}^\mathcal{F}_z = \set[\Big]{ \mu \suchthat u(z) \le \int u\,d\mu
    \enspace\text{for all $u \in \mathcal{F}$}},
\]
the so-called \emph{Jensen measures} for $\mathcal{F}$ with barycenter $z$.
Since $\mathcal{F}$ contains the constants, it follows that measures in
$\mathcal{J}^\mathcal{F}_z$ are probability measures. If $g\colon X \to
\mathbb{R} \cup \set{ +\infty }$, we define
\[
    I^{\mathcal{F}} g(z) = \inf \set[\Big]{ \int g\,d\mu \suchthat \mu \in J^{\mathcal{F}}_z }.
\]
Edwards' theorem shows that when $g$ is lower semicontinuous, $S^{\mathcal{F}} g
= I^{\mathcal{F}} g$. More precisely,

\begin{theorem}[Edwards' theorem]\label{thm:edwards}
    If $g\colon X \to \mathbb{R} \cup \set{ +\infty }$ is a Borel function on $X$,
    then $I^{\mathcal{F}} g \le S^{\mathcal{F}} g$. If additionally $g$
    is lower semicontinuous on $X$, then $I^{\mathcal{F}} g = S^{\mathcal{F}} g$.
\end{theorem}

For an accessible proof, we refer to~\cite{wikstrom}. Note that Edwards' theorem
in~\cite{wikstrom} is formulated for bounded lower semicontinuous functions $g$,
but the proof goes through without changes even if the condition on boundedness
is dropped.

In particular, we will apply Edwards' theorem for the two cones $\psh^-(\Omega)
= \mathcal{USC}(\bar\Omega) \cap \psh(\Omega)$ and $\psh^c(\Omega) =
C(\bar\Omega) \cap \psh(\Omega)$. We will denote the corresponding Jensen
measures by $\mathcal{J}^-_z(\Omega)$ and $\mathcal{J}^c_z(\Omega)$,
respectively.

It was shown in~\cite{wikstrom} (see also~\cite{dieu-wikstrom})
that if $\Omega$ is a bounded $B$-regular domain and $z \in \Omega$,
then $\mathcal{J}^-_z(\Omega) = \mathcal{J}^c_z(\Omega)$.

We can now formulate and prove our main result.
In order to use Edwards' theorem and the Jensen measures machinery,
we need to assume that $g$ is lower semicontinuous on $\bar\Omega$. In
fact, we will make the stronger assumption that $g^* = g_*$ on $\bar\Omega$,
i.e.\ that $g$ is a continuous on $\bar\Omega$ as an extended real-valued
function. This assumption guarantees that $S^{\mathcal{F}} g$ is plurisubharmonic
(without upper semicontinuous regularization).

\begin{theorem}\label{thm:dirichlet-problem}
    Assume that $\Omega \subset \mathbb{C}^n$ is a bounded $B$-regular domain, and
    $\phi$ is a lower bounded, tame plurisuperharmonic function on $\Omega$ such
    that $\phi^* = \phi_*$ on $\bar\Omega$.
    Then the associated Perron--Bremermann envelope,
    \[
        P\phi(z) = \sup \set{ u(z) \suchthat u \in \psh(\Omega), u^* \le \phi_* }
    \]
    is a maximal plurisubharmonic function that is continuous outside a pluripolar set.
    Furthermore, $P\phi$ is the unique maximal plurisubharmonic function with the correct boundary values, i.e.\ for $z_0 \in \bd\Omega,$
    \[
        \lim_{\Omega \ni \zeta \rightarrow z_0} P\phi(\zeta) = \phi(z_0),
    \]
    where the limit is taken in the extended real sense, and $\phi(z_0) = +\infty$ is allowed.
\end{theorem}

\begin{proof}
    Note that we may assume that $\phi \ge 0$.
    Since $P\phi \le \phi$ on $\Omega$, it follows that on $\bar\Omega$,
    \[
        (P\phi)^* \le \phi^*
    \]
    so $(P\phi)^*$ is a plurisubharmonic function in the defining family
    of $P\phi$, so $(P\phi)^* \le P\phi$. Conversely, it is clear that
    $(P\phi)^* \ge P\phi$, so $P\phi = (P\phi)^*$ is
    plurisubharmonic.

    Maximality of $P\phi$ is standard, but for completeness, we include the
    argument. Let $\Omega' \Subset \Omega$ and assume that $u \le P\phi$ on
    $\Omega' \setminus \Omega$. Then
    \[
        \tilde u = \begin{cases}
            P\phi, & z \in \Omega' \setminus \Omega \\
            \max \set{ u, P\phi}, & z \in \Omega'
        \end{cases}
    \]
    is plurisubharmonic on $\Omega$ and $\tilde u \le \phi$, so $\tilde u \le P\phi$, i.e. $u \le P\phi$ everywhere on $\Omega$, which shows that $P\phi$ is maximal.

    Furthermore, since $P\phi$ is majorized by the tame function $\phi$, it follows
    from Proposition~\ref{prop:basic} and Theorem~\ref{thm:qb-psh} that $P\phi$ is
    tame, and thus quasibounded, and consequently so are all the functions in the defining
    family for $P\phi$. Hence, using Theorem~\ref{thm:edwards} and the equality of Jensen measures
    for upper bounded plurisubharmonic functions and continuous plurisubharmonic
    functions~\cite{dieu-wikstrom, wikstrom} on $B$-regular domains, it follows that for $z \in \Omega$,
    \begin{align*}
        P\phi(z) &= \sup \set{ u(z) \suchthat u \in \psh(\Omega), u^* \le \phi_* } \\
              &= \big( \sup \set{ u(z) \suchthat u \in \psh^-(\Omega), u^* \le \phi_* } \big)^* \\
              &= \Big( \inf \set{ \int \phi_*\,d\mu \suchthat \mu \in \mathcal{J}_z^-(\Omega) } \Big)^* \\
              &= \Big( \inf \set{ \int \phi_*\,d\mu \suchthat \mu \in \mathcal{J}_z^c(\Omega) } \Big)^* \\
              &= \big( \sup \set{ u(z) \suchthat u \in \psh(\Omega)  \cap C(\bar\Omega), u^* \le \phi_* } \big)^*.
    \end{align*}
    Thus, $P\phi$ coincides outside of a pluripolar set with
    a lower semicontinuous function (namely the supremum of a family of continuous functions)
    and since $P\phi$ is plurisubharmonic, it is also
    upper semicontinuous. This establishes the claim of continuity outside
    a pluripolar set.

    We now show the statement about the boundary values of $P\phi$.
    Let $\phi_n$ be a sequence of continuous plurisuperharmonic
    functions such that $\phi_n \nearrow \phi$ on $\overline{\Omega}$.
    Then $P \phi_n$ are continuous functions up to the boundary, as solutions
    to the generalized Dirichlet problem with regards to $\Omega$,
    see B\l{}ocki~\cite{blocki}.

    Now take $z_0 \in \bd \Omega$. If $\phi(z_0) < +\infty$, for every
    $\eps > 0$, there exists
    an~$n$ such that $\phi_n(z_0) > \phi(z_0) - \eps$. We then have
    \[
        \phi(z_0) - \eps < \phi_n(z_0) =
        \liminf_{\Omega \ni \zeta \rightarrow z_0} P\phi_n(\zeta) \leq
        \liminf_{\Omega \ni \zeta \rightarrow z_0} P\phi(\zeta),
    \]
    yet
    \[
    \limsup_{\Omega \ni \zeta \rightarrow z_0} P\phi(\zeta) \leq \limsup_{\Omega \ni \zeta \rightarrow z_0} \phi(\zeta) = \phi(z_0).
    \]
    Since $\eps$ was chosen arbitrarily, it follows that
    $\lim_{\Omega \ni \zeta \rightarrow z_0} P\phi(\zeta) = \phi(z_0)$.
    If on the other hand, $\phi(z_0)=\infty$, then
    $\lim_{\Omega \ni \zeta \rightarrow z_0} P\phi_n(\zeta)$ may be made
    arbitrarily large. It follows that $\liminf_{\Omega \ni \zeta \rightarrow z_0}
    P\phi(\zeta)= \infty$, which implies that
        \[
        \lim_{\Omega \ni \zeta \rightarrow z_0} P\phi(\zeta) = \infty = \phi(z_0)
    \]
    in the extended reals.

    It remains to show the statement about uniqueness. Suppose that $v$ is another maximal plurisubharmonic function with the correct boundary values. Then, $v \leq P\phi$, because $v$ is a member in the defining family for $P\phi$. Since $P\phi$ may be approximated from below outside a pluripolar set by functions in $\psh(\Omega)  \cap C(\bar\Omega)$, there exists a point $z_0 \in \Omega$, a bounded continuous maximal plurisubharmonic function $u_b\leq \phi$ and $k>0$ such that $v(z_0)+k < u_b(z_0)$. In particular, the open set
    \[
    E=\{z \in \Omega : v(z)-u_b(z) < -k\}
    \]
    is non-empty, and the boundary values of $v$ guarantees that $E \Subset \Omega$. Thus we may find an open set $\tilde \Omega$ such that $E \Subset \tilde \Omega \Subset \Omega$ and $v \geq u_b - k$ on $\bd \tilde \Omega$. Since both $v$ and $u_b - k$ are maximal and belong to $L^\infty(\tilde \Omega)$, the comparison principle yields in particular $v(z_0) +k \geq u_b(z_0)$, a contradiction.
\end{proof}
\begin{remark}
One might argue that it would be more natural to drop the assumption
that $\phi_* = \phi^*$ and consider the ``free'' Perron--Bremermann
function,
\[
    \tilde P\phi(z) = \sup \set{ u(z) \suchthat u \in \psh(\Omega), u \le \phi },
\]
but note that with the assumption that $\phi_* = \phi^*$, it follows that if
$u \le \phi$ on $\Omega$, then $u^* \le \phi^* = \phi_*$ on $\bar\Omega$, so
in this case (for $z \in \Omega$), $P\phi = \tilde P\phi$.

Without the continuity assumptions, things can be quite different. In general,
the free Perron--Bremermann might not be plurisubharmonic without upper
semicontinuous regularization, and even if it is, it might not attain the
correct boundary values. Also, in general $\tilde P \phi > P\phi_*$.
To see this, take $\phi = u$ as in Example~\ref{ex:poisson-kernel}
as the Poisson kernel on the unit disc in $\mathbb{C}$. Then $\tilde P\phi = \phi$,
but $P\phi_* \equiv 0$.

We would also like to point out that if
\[
    \sup \set{ u(z) \suchthat u \in \psh^-(\Omega), u^* \le \phi_* }
\]
is plurisubharmonic (without upper semicontinuous regularization), it would
follow that $P\phi$ is continuous everywhere on $\bar\Omega$. In all the examples
below, this is the case, and we are not aware of any examples where
the upper semicontinuous regularization is actually needed, assuming $\phi$ is continuous.

\end{remark}

Even if the plurisuperharmonic function $\phi$ is not tame, it is always possible to weaken
its singularities enough for Theorem~\ref{thm:dirichlet-problem} to apply. In this situation,
it is also possible to specify a set on which the solution is guaranteed to be continuous.

\begin{theorem}
    Assume that $\Omega \subset \mathbb{C}^n$ is a bounded $B$-regular domain, and that
    $\phi$ is a lower bounded plurisuperharmonic function on $\Omega$ such
    that $\phi^* = \phi_*$ on $\bar\Omega$.
    Then for any increasing concave function $\psi$ defined on an interval $[a,+\infty]$ containing the range of $\phi$, such that
    $$ \lim_{t\rightarrow \infty} \frac{\psi(t)}{t}\rightarrow 0,$$
    $P(\psi\circ \phi)$ is continuous on $\Omega \setminus  \{z \in \Omega : \phi(z)=\infty\}$, and constitutes the unique solution to the generalized Dirichlet problem associated to $\psi \circ \phi$.
\end{theorem}

\begin{proof}
    Note that if either $\phi$ or $\psi$ is bounded, $\psi \circ \phi$ is bounded and continuous
    on the boundary, and so the result follows as an instant of the generalized Dirichlet problem
    with continuous boundary data, see B\l{}ocki~\cite{blocki}. Further, adding constants if necessary,
    we may without loss of generality assume that $\phi\geq0$ and that $\psi$ is a bijection from
    $[0, +\infty]$ to $[0, +\infty]$. As $\psi\circ \phi$ is plurisuperharmonic, and $\psi^{-1}$ is
    an increasing convex function, it follows from Theorem~\ref{thm:qb-psh2} and
    Theorem~\ref{thm:dirichlet-problem} that $P(\psi\circ \phi)$ is continuous quasi-everywhere,
    and constitutes the unique solution to the generalized Dirichlet problem associated to $\psi \circ \phi$.

    It remains to show that $P(\psi \circ \phi) \in C(\Omega\setminus  \{z \in \Omega : \phi(z)=\infty\})$.
    Clearly, $\phi$ is a plurisuperharmonic majorant to $\psi^{-1}\circ P(\psi \circ \phi)$. Reasoning
    as in the proof of~\eqref{eq:condition-phi} of Theorem~\ref{thm:qb-psh2}, this implies that for every
    fixed $\eps>0$, there is a positive integer $N$ such that
    \[
        P(\psi\circ \phi) \le \eps\psi^{-1}\circ P (\psi\circ \phi) + N \le \eps \phi + N.
    \]
    In particular,
    \[
        u_{\eps} := \max\{P(\psi\circ \phi) - \eps \phi,0 \}
    \]
    defines a sequence of bounded plurisubharmonic functions converging to $P(\psi\circ \phi)$ at all points outside the set $\{z \in \Omega : \phi(z)=\infty\}$ as $\eps \rightarrow 0$, and continuity follows.
\end{proof}
\begin{corollary}
    Assume that $\Omega \subset \mathbb{C}^n$ is a bounded $B$-regular domain, and
    that $\phi \ge 0$ is a plurisuperharmonic function on $\Omega$ such that $\phi^*=\phi_*$ on $\overline{\Omega}$. Then, for
    every $0 < \alpha < 1$, the Perron--Bremermann envelope of $\phi^\alpha$ is
    continuous on $\Omega \setminus \{z \in \Omega : \phi(z)=\infty\}$, and constitutes the unique solution to the corresponding Dirichlet problem. The same holds for the Perron--Bremermann envelope
    of $\log^+ \phi$ and $\phi^\alpha \log^+ \phi$.
\end{corollary}

As previously noted, the proof of Theorem~\ref{thm:dirichlet-problem} shows that if $P\phi$
is quasibounded (even if $\phi$ itself is not tame), then $P\phi \in C(\Omega)$,
but currently we are not aware of any direct conditions that ensure
quasiboundedness of $P\phi$. Indeed, as the following two examples illustrate, 
$P\phi$ may very well be quasibounded, even when $\phi$ fails to be tame.

\begin{example}
    Let $\Omega = \mathbb{B}$ be the unit ball in $\mathbb{C}^2$, let $0 <
    \alpha < 1$, and let $\phi(z, w) = (- \log\abs{z})^\alpha$. Then $\phi$ is
    tame, so $P\phi$ is the unique solution to the Dirichlet problem. In fact,
    \[
        P\phi(z) = \Big(-\frac12 \log(1-\abs{w}^2) \Big)^\alpha.
    \]
    To obtain the explicit formula for $P\phi$, fix $w \in \mathbb{D}$ and
    consider the analytic disc $f\colon \mathbb{D} \to \mathbb{B}$ where
    $f(\zeta) = \big(\zeta \sqrt{1-\abs{w}^2}, w \big)$. Then
    $\phi(f(e^{i\theta})) = \big(-\frac12 \log(1-\abs{w}^2)\big)^\alpha$ for
    every $\theta$, so $\phi$ is constant on $f(\partial\mathbb{D})$. By the
    maximum principle, if $u \le \phi$ is plurisubharmonic, then
    \[
        u(f(\zeta),w) \le \Big(-\frac12 \log(1-\abs{w}^2)\Big)^\alpha.
    \]
    Varying $w$ and $\zeta$, we see that
    \[
        u(z,w) \le \Big(-\frac12 \log(1-\abs{w}^2)\Big)^\alpha,
    \]
    for every $(z, w) \in \mathbb{B}$ and taking the supremum over $u$,
    \[
        P\phi(z,w) \le U:= \Big(-\frac12 \log(1-\abs{w}^2)\Big)^\alpha.
    \]
    A direct computation shows that
    \[
        \frac{\partial^2 U}{\partial w \partial \bar w} =
        \frac{\alpha \big(-\log(1-\abs{w}^2)\big)^\alpha \big( (\alpha-1)\abs{w}^2 - \log(1-\abs{w}^2)\big)}{(1-\abs{w}^2)^2 \big(\log(1-\abs{w}^2)\big)^2}
    \]
    and all the factors, except possibly $q(w) = (\alpha-1)\abs{w}^2 -
    \log(1-\abs{w}^2)$ are non-negative for $w\in\mathbb{D}$. It is a
    straight-forward calculus exercise to show that $ct - \log(1-t) \ge 0$ for
    $0 < t < 1$ when $c \ge -1$, so in fact $q(w) \ge 0$ as well. Hence $U$
    is plurisubharmonic on $\mathbb{B}$ and thus $P\phi \ge U$. So, in fact
    \[
        P\phi(z,w) = U(z,w) = \Big( -\frac12 \log(1-\abs{w}^2) \Big)^\alpha.
    \]
    In the case $\alpha = 1$, the above computation goes through without any
    changes, but in this case, $\phi$ is \emph{not} tame and
    Theorem~\ref{thm:dirichlet-problem} does not apply.

    To see that $\phi(z,w) = -\log\abs{z}$ is not tame, take any nonnegative function $\psi$
    satisfying~\eqref{eq:condition-phi}. If $v$ is a plurisuperharmonic
    majorant of $\psi \circ \phi$, then $u = -v$ is a plurisubharmonic
    minorant of $-\psi \circ \phi$. Hence, the Lelong number of $u$ at $z=0$
    is
    \begin{align*}
        \nu_u(0,0) &= \liminf_{(z,w) \to (0,0)} \frac{2u(z,w)}{\log(\abs{z}^2 + \abs{w}^2)} \\
        &\ge \liminf_{(z,w) \to (0,0)} \frac{-2\psi(-\log\abs{z})}{\log(\abs{z}^2 + \abs{w}^2)} \\
        &= \liminf_{(z,w) \to (0,0)} \frac{\psi(-\log\abs{z})}{-\log\abs{z}}
        \frac{\log\abs{z}^2}{\log(\abs{z}^2 + \abs{w}^2)} = \infty,
    \end{align*}
    which is a contradiction and Theorem~\ref{thm:qb-psh} shows that $\phi$ is not tame.

    On the other hand,
    \[
        P\phi(z,w) = -\frac12 \log(1-\abs{w}^2),
    \]
    which is clearly continuous on $\mathbb{B}$. Furthermore
    \[
        P\phi(z,w) = -\frac12 \log(1-\abs{w}^2) =
        \sup_{r > 1} \set[\Big]{-\frac12 \log(1-r\abs{w}^2)}
    \]
    which shows that $P\phi$ is quasibounded. 
\end{example}
\begin{example}
    Let $\Omega = \mathbb{B}$ be the unit ball in $\mathbb{C}^2$, let $0 <
    \alpha < 1$, and let $\phi(z, w) = (- \log\abs{zw})^\alpha$. Again, $\phi$ is
    tame, and  $P\phi(z)$ may explicitly be written as
    \[
        P\phi(z)=
        \begin{cases}
            \Big(-\frac12 \log(\abs{z}^2-\abs{z}^4) \Big)^\alpha  &\abs{z} \geq \frac{\sqrt{2}}{2} \\
            \Big(-\frac12 \log(\abs{w}^2-\abs{w}^4) \Big)^\alpha  &\abs{w} \geq \frac{\sqrt{2}}{2} \\
            (\log 2)^\alpha  &\abs{z},\abs{w} \leq \frac{\sqrt{2}}{2}.

        \end{cases}
    \]
    To see this, note that on the boundary, $\phi$ attains the value $(\log 2)^\alpha$ at $|z|=|w|=\frac{\sqrt{2}}{2}$, so it is clear that $P\phi$ is bounded from above by $(\log 2)^\alpha$ on the bidisk $\{z,w \in \mathbb{C}^2: \abs{z}, \abs{w} \leq \frac{\sqrt{2}}{2}\}$. $P\phi$ is also bounded from above by the other two cases above, as is easily seen as in the previous example using the analytic discs
    \[
    f_1(\zeta) = \big(\zeta \sqrt{1-\abs{w}^2}, w \big),\ f_2(\zeta) = \big(z,\zeta \sqrt{1-\abs{z}^2} \big),
    \]
    fixing $w,z \in \mathbb{D}$ respectively. We now check that the above function
    is indeed a plurisubharmonic function. Due to symmetry and a standard gluing
    lemma for plurisubharmonic functions, it is enough to consider points where
    $\abs{z}>\frac{\sqrt{2}}{2}$ to see that the Laplacian is non-negative in
    the sense of distributions. We have
    \[
        \frac{\partial^2 P\phi}{\partial z \partial \bar z}
        = \alpha\abs{z}^2 \big(-\log(\abs{z}^2-\abs{z}^4)\big)^\alpha\cdot \frac{(\alpha-1)(1-2\abs{z}^2)^2-\abs{z}^2\log(\abs{z}^2-\abs{z}^4)}{(\abs{z}^2-\abs{z}^4)^2(-\log(\abs{z}^2-\abs{z}^4)^2},
    \]
    where all factors are positive except possibly $(\alpha-1)(1-2\abs{z}^2)^2-\abs{z}^2\log(\abs{z}^2-\abs{z}^4)$. Again, it is straightforward to check that indeed, $(\alpha-1)(1-2t)^2-t\log(t-t^2)>0$ for $0 <
    \alpha < 1$, $t> \frac{1}{2}$, so $P\phi$ is plurisubharmonic.

    Similar calculations show that the above formula for $P\phi$ remains true when $\alpha = 1$. However, using the fact that  $-\log\abs{z}\leq -\log\abs{zw}$ together with property (3) of Proposition~\ref{prop:basic}, we have that
    \[
     S(-\log\abs{z}) \leq S(-\log\abs{zw}),
    \]
    which implies that  $-\log\abs{zw}$ is not tame.
\end{example}
In both these examples, locally, $P(\phi^\alpha)$ depends on one variable only. Thus, this is true for the plurisubharmonic function $P(\phi^\alpha)^{\frac{1}{\alpha}}$ as well, which then constitutes a solution to the generalized Dirichlet problem associated to $\phi$. We collect this observation in the following proposition.
\begin{proposition}
     Assume that $\Omega \subset \mathbb{C}^n$ is a bounded $B$-regular domain, and
    that $\phi\geq 0$ is a plurisuperharmonic function on $\Omega$ such that $\phi^*=\phi_*$ on $\overline{\Omega}$. Further assume that there exist a continuous, concave, strictly increasing function $\psi\colon [0,+\infty) \to [0,+\infty)$ 
 such that $\psi\circ \phi$ is tame and $P(\psi\circ \phi)$ locally depends on at most $n-1$ variables. Then $\psi^{-1}\circ P(\psi\circ \phi)$ is a quasi-everywhere continuous solution to the generalized Dirichlet problem with respect to $\phi$ on $\Omega$.
\end{proposition}

The proof of Theorem~\ref{thm:dirichlet-problem} also shows that if $P\phi$
is quasibounded \textit{quasi-everywhere} (even if $\phi$ itself is not tame), then $P\phi$ is the unique solution. In our setting, these are in fact equivalent statements, as the following theorem shows.

\begin{theorem}
Assume that $\Omega \subset \mathbb{C}^n$ is a bounded $B$-regular domain, and
    that $\phi \ge 0$ is a plurisuperharmonic function on $\Omega$ such that $\phi^*=\phi_*$ on $\overline{\Omega}$. Then every $v$ solving the associated Dirichlet problem satisfies
    $$ \big(\lim_{n\rightarrow \infty} P(\min\{n, \phi\})\big)^* \leq v \leq P \phi.$$
    In particular, $P\phi$ is quasibounded quasi-everywhere if and only if $P\phi$ uniquely solves the associated Dirichlet problem.
\end{theorem}
\begin{proof}
    As $h$ is a plurisubharmonic function in the defining family for $P\phi$, it is clear that $h \leq P\phi$. To show the second inequality, assume that there exists a point $z_0 \in \Omega$ where it does not hold. This implies that there exist $n,k>0$ such that $v(z_0)+k < P(\min\{n, \phi(z_0)\})$, and reasoning as in the proof of Theorem~\ref{thm:dirichlet-problem} concerning uniqueness, we reach a contradiction.

    To show the last statement, note that every bounded plurisubharmonic function is bounded by $P(\min\{n, \phi\})$ for some $n$, and so the regularization of the envelope over bounded plurisubharmonic functions coincides with $\big(\lim_{n\rightarrow \infty} P(\min\{n, \phi\})\big)^*$, which is a maximal plurisubharmonic function as the complex Monge--Ampère operator is continuous along increasing sequences. Again reasoning as in the proof of Theorem~\ref{thm:dirichlet-problem}, $\big(\lim_{n\rightarrow \infty} P(\min\{n, \phi\})\big)^*$ also has the correct boundary limits, and thus solves the associated Dirichlet problem.
\end{proof}
We end with an example of when the two envelopes do not coincide, illustrating the difficulty of addressing the question of uniqueness for (unbounded) maximal plurisubharmonic functions with given boundary values in full generality.

\begin{example}
    Let $\Omega = \mathbb{B}$ be the unit ball in $\mathbb{C}^2$. Define
    \[
        u(z,w)=\frac{1-\abs{z}^2}{\abs{1-z}^2}, \ v(z,w)=\frac{\abs{w}^2}{\abs{1-z}^2}.
    \]
    As was shown in connection to Example 1.2, $u$ is a maximal plurisubharmonic function, but not quasibounded. In fact, $v$ is a maximal plurisubharmonic function as well, as
    \[
        \frac{\partial^2 v}{\partial z \partial \bar z}+\frac{\partial^2 v}{\partial w \partial \bar w} = \frac{\abs{1-z}^2 + \abs{w}^2}{\abs{1-z}^4} \geq 0,
    \]
    and
    \[
        \frac{\partial^2 v}{\partial z \partial \bar z}\cdot \frac{\partial^2 v}{\partial w \partial \bar w}- \frac{\partial^2 v}{\partial z \partial \bar w}\cdot \frac{\partial^2 v}{\partial w \partial \bar z}=0.
    \]
    Note that $u \geq v$. Now define $v_\eps := \frac{\abs{w}^2}{\abs{1+\eps-z}^2}$. It is straight-forward to check that these are maximal plurisubharmonic functions as well, and that
    \[
        v = \sup_{\eps>0} v_\eps,
    \]
    which shows that $v(z,w)$ is quasibounded. Furthermore, $v$ is constant on the analytic discs $f_k(\zeta)=(k\zeta,k(\zeta-1))$, and in particular, $v \circ f_k$ is harmonic. This implies that $v$ is larger than any bounded plurisubharmonic function smaller than $u$, and thus equals the Perron--Bremermann envelope of $u$ over $\psh^-(\Omega)$.

    Finally, note that for all $(z_0,w_0) \in \bd \Omega$, we have that
    \begin{align*}
        \limsup_{(z,w) \to (z_0,w_0)} u(z,w) & =\limsup_{(z,w) \to (z_0,w_0)} v(z,w) \\
        \liminf_{(z,w) \to (z_0,w_0)} u(z,w) &= \liminf_{(z,w) \to (z_0,w_0)} v(z,w).
    \end{align*}
    Indeed, if $z_0\neq 1$, the above equalities clearly hold, and at the boundary point $(z_0,w_0)=(1,0)$, we have
    \begin{align*}
        \limsup_{(z,w) \to (1,0)} u(z,w) & =\limsup_{(z,w) \to (1,0)} v(z,w) =\infty \\
        \liminf_{(z,w) \to (1,0)} u(z,w) &= \liminf_{(z,w) \to (1,0)} v(z,w)=0.
    \end{align*}
    To get equality everywhere on the boundary, we can simply add the pluriharmonic function $-\log\abs{z-1}$ to both $u$ and $v$. Specificially,
    let $\tilde u(z,w) = u(z,w) - \log\abs{z-1}$ and
    $\tilde v(z,w) = v(z,w) - \log|z-1|$. Then $\tilde u$ and $\tilde v$
    are two maximal plurisubharmonic functions continuous on $\bar\Omega$
    (as extended real valued functions) that coincide on $\bd \Omega$.
\end{example}

\end{document}